\documentclass[british]{amsart}

\usepackage{amstext}
\usepackage{amsthm}

\makeatletter
\numberwithin{equation}{section}
\numberwithin{figure}{section}
\theoremstyle{plain}
\newtheorem{thm}{\protect\theoremname}[section]
\theoremstyle{plain}
\newtheorem{prop}[thm]{\protect\propositionname}
\theoremstyle{plain}
\newtheorem{lem}[thm]{\protect\lemmaname}

\usepackage{alex}

\makeatother

\usepackage{babel}
\addto\captionsbritish{\renewcommand{\lemmaname}{Lemma}}
\addto\captionsbritish{\renewcommand{\propositionname}{Proposition}}
\addto\captionsbritish{\renewcommand{\theoremname}{Theorem}}
\addto\captionsenglish{\renewcommand{\lemmaname}{Lemma}}
\addto\captionsenglish{\renewcommand{\propositionname}{Proposition}}
\addto\captionsenglish{\renewcommand{\theoremname}{Theorem}}
\providecommand{\lemmaname}{Lemma}
\providecommand{\propositionname}{Proposition}
\providecommand{\theoremname}{Theorem}

\begin{document}
\title{On Fillmore's theorem over integrally closed domains}
\author{Alexander Stasinski}
\begin{abstract}
A well-known theorem of Fillmore says that if $A\in\M_{n}(K)$ is
a non-scalar matrix over a field $K$ and $\gamma_{1},\dots,\gamma_{n}\in K$
are such that $\gamma_{1}+\dots+\gamma_{n}=\Tr(A)$, then $A$ is
$K$-similar to a matrix with diagonal $(\gamma_{1},\dots,\gamma_{n})$.
Building on work of Borobia, Tan extended this by proving that if
$R$ is a unique factorisation domain with field of fractions $K$
and $A\in\M_{n}(R)$ is non-scalar, then $A$ is $K$-similar to a
matrix in $\M_{n}(R)$ with diagonal $(\gamma_{1},\dots,\gamma_{n})$.
We note that Tan's argument actually works when $R$ is any integrally
closed domain and show that the result cannot be generalised further
by giving an example of a matrix over a non-integrally closed domain
for which the result fails. 

Moreover, Tan gave a necessary condition for $A\in\M_{n}(R)$ to be
$R$-similar to a matrix with diagonal $(\gamma_{1},\dots,\gamma_{n})$.
We show that when $R$ is a PID and $n\geq3$, Tan's condition is
also sufficient.
\end{abstract}

\address{\selectlanguage{english}%
Department of Mathematical Sciences, Durham University, Durham, DH1
3LE, UK}
\email{\selectlanguage{english}%
alexander.stasinski@durham.ac.uk}
\maketitle

\section{Introduction}

For a commutative ring $R$, we say that $A,B\in\M_{n}(R)$ are \emph{$R$-similar}
if there exists a $g\in\GL_{n}(R)$ such that $gAg^{-1}=B$. Fillmore's
theorem on matrices over fields \cite[Theorem~2]{Fillmore} is the
statement that for any non-scalar matrix $A\in\M_{n}(K)$ over a field
$K$ and any $\gamma_{1},\dots,\gamma_{n}\in K$ such that $\gamma_{1}+\dots+\gamma_{n}=\Tr(A)$,
$A$ is $K$-similar to a matrix with diagonal $(\lambda_{1},\dots,\lambda_{n})$.
Despite its simple proof, Fillmore's theorem is a widely cited result
in matrix theory and has by now found its way into textbooks \cite[Theorem~1.5]{Zhan-Matrix-theory}.

In \cite{Borobia_Fillmore}, Borobia proved a variant of Fillmore's
theorem for matrices over $\Z$, namely that for any non-scalar $A\in\M_{n}(\Z)$
and any $\gamma_{1},\dots,\gamma_{n}\in\Z$ such that $\gamma_{1}+\dots+\gamma_{n}=\Tr(A)$,
$A$ is $\Q$-similar to a matrix $\M_{n}(\Z)$ with diagonal $(\lambda_{1},\dots,\lambda_{n})$.
Note that this is a statement about $\Q$-similarity, not $\Z$-similarity.
Borobia's proof in fact goes through essentially word for word for
any PID $R$ instead of $\Z$ (with $\Q$ replaced by the field of
fractions of $R$) and as such is a generalisation of Fillmore's theorem.
The main step in Borobia's proof is to show that any non-scalar matrix
$A\in\M_{n}(\Z)$ is $\Q$-similar to a matrix in $\M_{n}(\Z)$ with
an off-diagonal entry equal to $1$. Generalising this, Tan \cite{Tan_Fillmore}
proved that if $\Z$ is replaced by a unique factorisation domain
(UFD), aka a factorial ring, $R$ and $\Q$ by the field of fractions
$K$, then the rational canonical form implies that any non-diagonal
$A\in\M_{n}(R)$ is $K$-similar to a matrix in $\M_{n}(R)$ with
an off-diagonal entry equal to $1$. This holds because the characteristic
polynomial $\chi_{A}(x)$ of $A$ is in $R[x]$ and thus, since
$R$ is a UFD, any monic factor of $\chi_{A}(x)$ (considered as polynomials
in $K[x]$), lies in $R[x]$. Therefore every invariant factor of
$A$ is in $R[x]$ and hence the rational canonical form of $A$ (under
$K$-similarity) is in $\M_{n}(R)$.

The purpose of this note is twofold. First, we record the fact that
Tan's argument goes through when $R$ is any integrally closed domain,
because under this assumption, it is still true that any monic factor
of $\chi_{A}(x)$ lies in $R[x]$. Nevertheless, it is not immediately
obvious whether the Borobia--Tan generalisation of Fillmore's theorem
holds when $R$ is an integral domain that is not integrally closed.
We give a counter-example when $R=\Z[\sqrt[3]{16}]$, showing that
the result cannot be generalised further.

Our second aim is concerned with another generalisation of Fillmore's
theorem. The Borobia-Tan generalisation allows $K$-similarity, where
$K$ is the field of fractions of an integrally closed domain $R$, as long as
the resulting matrix has all its entries in $R$. A perhaps more natural,
and much more difficult, question is to what extent Fillmore's theorem
holds when only $R$-similarity is allowed, that is, given an integrally closed domain and a non-scalar matrix $A=(a_{ij})\in\M_{n}(R)$,
when is it true that for all elements $\gamma_{1},\dots,\gamma_{n}\in R$
such that $\gamma_{1}+\dots+\gamma_{n}=\Tr(A)$, $A$ is $R$-similar
to a matrix with diagonal $(\gamma_{1},\dots,\gamma_{n})$? This question
was raised, for $R$ a UFD, by Tan \cite[Remark~2.3]{Tan_Fillmore}.
Very little is known about $R$-similarity of matrices in this generality, but when $R$
is a PID, we show that the question has a positive answer as soon
as $n\geq3$ and $A$ is not scalar modulo any proper ideal of $R$
(these conditions are seen to be necessary). This answers Tan's question
when $R$ is a PID.

\section{The Borobia-Tan generalisation of Fillmore's theorem}
\begin{thm}
\label{thm:=000020main}Let $R$ be an integrally closed domain with
field of fractions $K$. Let $A\in\M_{n}(R)$ with $n\geq2$ be non-scalar.
Then, for any $\gamma_{1},\dots,\gamma_{n}\in R$ such that $\gamma_{1}+\dots+\gamma_{n}=\Tr(A)$,
there exists a $B\in\M_{n}(R)$ with diagonal $(\gamma_{1},\dots,\gamma_{n})$
such that $A$ is $K$-similar to $B$. 
\end{thm}

\begin{proof}
The proof is essentially identical to Tan's \cite[Section~2]{Tan_Fillmore},
but noting that the rational canonical form (under $K$-similarity)
of $A$ lies in $\M_{n}(R)$, because every monic factor of the characteristic
polynomial of $A$ lies in $R[x]$ when $R$ is integrally closed
(see \cite[Ch.~3, \S3, Theorem~5]{Zariski-Samuel-I}).
\end{proof}
We now show that the above theorem is the most general possible in
the sense that there exists a noetherian domain of dimension $1$
for which the statement fails.

Let $\alpha=\sqrt[3]{16}$ and $K=\Q(\alpha)=\Q(\sqrt[3]{2})$, so
that $K$ is the field of fractions of $\Z[\alpha]$. Since $(\frac{\alpha}{2})^{3}=2$,
the ring $\Z[\alpha]$ is not integrally closed. It is however a noetherian
domain of dimension $1$ (see \cite[(2.12)~Proposition]{Neukirch}).
Define the matrix

\[
A=\begin{pmatrix}0 & 2 & \alpha\\
2\alpha & 0 & 4\\
4 & \alpha & 0
\end{pmatrix}\in\M_{3}(\Z[\alpha]).
\]
It is checked by direct computation that $A$ satisfies the polynomial
\[
X^{2}-\frac{\alpha}{2}X-8\alpha\in K[X],
\]
which is therefore the minimal polynomial of $A$ over $K$. The matrix
$A$ appears in \cite[Section~2.1]{Brewer-Richman} and is a variant
of an earlier example in the ring $\F_{2}[[t^{2},t^{3}]]$, due to
W.~C.~Brown. The original purpose of these examples was to find
matrices over an integral domain $R$ whose minimal polynomial (over
the field of fractions) does not have all its coefficients in $R$.
\begin{prop}
The matrix $A$ is not $K$-similar to any matrix in $\M_{3}(\Z[\alpha])$
with diagonal $(1,-1,0)$. Thus Theorem~\ref{thm:=000020main} does
not hold for $R=\Z[\alpha]$ and $A$ is not $K$-similar to any matrix
in $\M_{3}(\Z[\alpha])$ that has an off-diagonal entry equal to $1$.
\end{prop}

\begin{proof}
Assume that $A$ is $K$-similar to a matrix
\[
B=\begin{pmatrix}1 & a & b\\
c & -1 & d\\
e & f & 0
\end{pmatrix}\in\M_{2}(\Z[\alpha]).
\]
Then certainly $B$ has the same minimal polynomial as $A$ and thus
\begin{multline*}
B^{2}-\frac{\alpha}{2}B-8\alpha\\
=\begin{pmatrix}1+ac+eb-8\alpha-\frac{\alpha^{2}}{2} & bf-\frac{\alpha^{2}}{2}a & b-\frac{\alpha^{2}}{2}b+ad\\
ed-\frac{\alpha^{2}}{2}c & ac+1+df-8\alpha+\frac{\alpha^{2}}{2} & -d-\frac{\alpha^{2}}{2}d+bc\\
cf+e-\frac{\alpha^{2}}{2}e & -f-\frac{\alpha^{2}}{2}f+ea & df+eb-8\alpha
\end{pmatrix}=0.
\end{multline*}
Considering the diagonal entries, we obtain
\[
1+ac+eb-8\alpha-\frac{\alpha^{2}}{2}=ac+1+df-8\alpha+\frac{\alpha^{2}}{2}=df+eb-8\alpha=0,
\]
so that 
\[
eb=4\alpha+\frac{\alpha^{2}}{2}\qquad\text{and}\qquad df=4\alpha-\frac{\alpha^{2}}{2}.
\]
Since $a,b\in\Z[\alpha]$ and $\frac{\alpha^{2}}{2}\not\in\Z[\alpha]$,
this is a contradiction. The second statement now follows from Lemmas
1.2 and 1.3 in \cite{Tan_Fillmore}.
\end{proof}

\section{The $R$-similarity generalisation of Fillmore's theorem over a PID}

We now consider the question of generalising Fillmore's theorem to
matrices over a ring $R$ where only $R$-similarity is allowed. Tan
\cite[Theorem~2.2]{Tan_Fillmore} proved that for a non-scalar $A=(a_{ij})\in\M_{n}(R)$
to be $R$-similar to a matrix with diagonal $(\gamma_{1},\dots,\gamma_{n})$
for all $\gamma_{i}\in R$ such that $\gamma_{1}+\dots+\gamma_{n}=\Tr(A)$,
it is necessary that the ideal
\[
\mfa=(a_{ii}-a_{jj},a_{ij}\mid1\leq i,j\leq n,\,i\neq j)
\]
is equal to the whole ring $R$.
We observe that this follows easily without explicit matrix computations:
\begin{lem}
Let $R$ be a commutative ring, $A=(a_{ij})\in\M_{n}(R)$ and $\mfa=(a_{ii}-a_{jj},a_{ij}\mid1\leq i,j\leq n,\,i\neq j)$.
Then the following two conditions are equivalent.
\begin{enumerate}
\item $\mfa=R$;
\item For all proper ideals $\mfm$ of $R$, the image of $A$ in $\M_{n}(R/\mfm)$
is not a scalar matrix.
\end{enumerate}
Moreover, if for all $\gamma_{1},\dots,\gamma_{n}\in R$ such that
$\gamma_{1}+\dots+\gamma_{n}=\Tr(A)$, $A$ is $R$-similar to a matrix
with diagonal $(\gamma_{1},\dots,\gamma_{n})$, then the second condition
holds.
\end{lem}

\begin{proof}
Assume that $\mfa=R$ and suppose that $\mfm$ is a proper ideal of
$R$ such that the image of $A$ in $\M_{n}(R/\mfm)$ is scalar. Then
$a_{ii}-a_{jj}\in\mfm$ and $a_{ij}\in\mfm$, for all $1\leq i,j\leq n$
such that $i\neq j$. Thus $\mfa\subseteq\mfm$, which is a contradiction.
Conversely, assume that $\mfa\neq R$. Then $\mfa$ is a proper ideal
of $R$ such that the image of $A$ in $\M_{n}(R/\mfa)$ is scalar,
so the second condition is false.

Moreover, assume that for all $\gamma_{1},\dots,\gamma_{n}\in R$
such that $\gamma_{1}+\dots+\gamma_{n}=\Tr(A)$, $A$ is $R$-similar
to a matrix with diagonal $(\gamma_{1},\dots,\gamma_{n})$. Then,
for all $a,b\in R$, $A$ is $R$-similar to a matrix with diagonal
$(a,\dots)$ and $(b,\dots)$, respectively. Let $\mfm$ be an ideal
of $R$ such that the image of $A$ in $\M_{n}(R/\mfm)$ is scalar
and let $\rho:\M_{n}(R)\rightarrow\M_{n}(R/\mfm)$ denote the homomorphism
induced by $R\rightarrow R/\mfm$. Then $\rho(gAg^{-1})=\rho(g)\rho(A)\rho(g)^{-1}=\rho(A)$
for all $g\in\GL_{n}(R)$. Thus, for all $a,b\in R$, $a\equiv b\mod{\mfm}$,
so $R/\mfm$ has only one element and therefore $\mfm$ is not proper.
\end{proof}
We now show that Tan's necessary condition is also sufficient in the
case when $R$ is a PID and $A\in\M_{n}(R)$ with $n\geq3$. The key
is the following matrix similarity result originally proved by Laffey
and Reams \cite{Laffey-Reams} over $\Z$ and generalised to the case
of PIDs in \cite[Theorem~5.6]{commutatorsPID-2016}.
\begin{lem}
\label{lem:LF-normalform}Let $R$ be a PID and $A\in\M_{n}(R)$ with
$n\geq3$ non-scalar. Then $A$ is similar to a matrix $B=(b_{ij})\in\M_{n}(R)$
such that the image of $B$ in $\M_{n}(R/(b_{12}))$ is scalar. 
\end{lem}

The following result is the $R$-similarity generalisation of Fillmore's
theorem for matrices over a PID. It is a generalisation of \cite[Proposition~5.7]{commutatorsPID-2016},
where the special case of a trace zero matrix being $R$-similar to
a matrix with zero diagonal was considered (this was proved earlier
by Laffey and Reams for $R=\Z$). 
\begin{thm}
\label{thm:main-PIDs}Let $R$ be a PID and $A\in\M_{n}(R)$ with
$n\geq3$. Assume that for any proper ideal $\mfm$ of $R$, the image
of $A$ in $\M_{n}(R/\mfm)$ is not a scalar matrix. Then, for any
$\gamma_{1},\dots,\gamma_{n}\in R$ such that $\gamma_{1}+\dots+\gamma_{n}=\Tr(A)$,
$A$ is $R$-similar to a matrix with diagonal $(\gamma_{1},\dots,\gamma_{n})$.
\end{thm}

\begin{proof}
The matrix $A$ is necessarily non-scalar, so by Lemma~\ref{lem:LF-normalform}
we may assume without loss of generality that the image of $A$ in
$\M_{n}(R/(a_{12}))$ is scalar. By hypothesis, this implies that
$(a_{12})=R$, so $a_{12}$ is a unit. After conjugating $A$ by the
diagonal matrix $(a_{12}^{-1},1,\dots,1)$ to make the $(1,2)$ entry
equal to $1$, the conclusion now follows from Lemmas 1.2 and 1.3
in \cite{Tan_Fillmore}.
\end{proof}
The hypothesis $n\geq3$ in the above theorem is necessary: 
\begin{prop}
The matrix $A=\begin{pmatrix}1 & 2\\
-3 & -1
\end{pmatrix}\in\M_{2}(\Z)$ is not $\Z$-similar to a matrix in $\M_{2}(\Z)$ with diagonal $(0,0)$. 
\end{prop}

\begin{proof}
Assume that $A$ is similar to a matrix in $\M_{2}(\Z)$ with diagonal
$(0,0)$. Then there exists a $g=\begin{pmatrix}x & y\\
z & w
\end{pmatrix}\in\GL_{2}(\Z)$ such that
\[
gA=\begin{pmatrix}0 & 1\\
-5 & 0
\end{pmatrix}g,
\]
that is
\[
\begin{pmatrix}x-3y & 2x-y\\
-3w+z & -w+2z
\end{pmatrix}=\begin{pmatrix}z & w\\
-5x & -5y
\end{pmatrix}.
\]
Thus $g=\begin{pmatrix}x & y\\
x-3y & 2x-y
\end{pmatrix}$ so that $\det(g)=2x^{2}-2xy+3y^{2}=(x-y)^{2}+x^{2}+2y^{2}$. Since
we must have $\det(g)=\pm1$, this implies that $y=0$ and $2x^{2}=1$,
which is a contradiction.
\end{proof}
Whether Theorem~\ref{thm:main-PIDs} holds for an integrally closed
domain $R$ that is not a PID, is an open problem.

\bibliographystyle{alex}
\bibliography{alex}

\end{document}